\documentclass{amsart}

\usepackage{amsmath}
\usepackage{amsfonts}
\usepackage{amssymb}
\usepackage{amscd}
\usepackage{amsthm}
\usepackage{setspace}
\usepackage{graphicx}
\usepackage{hyperref}
\usepackage[shortlabels]{enumitem}

\newtheorem*{mainthm}{Main Theorem}

\everymath{\displaystyle}

\newtheorem{theorem}{Theorem}[section]
\newtheorem{lemma}[theorem]{Lemma}
\newtheorem{corollary}[theorem]{Corollary}
\newtheorem{proposition}[theorem]{Proposition}

\newtheorem{claim}[theorem]{Claim}

\theoremstyle{definition}
\newtheorem{remark}[theorem]{Remark}

\newtheorem{assumption}[theorem]{Assumption}

\newcommand{\F}{\mathbb{F}}

\newcommand{\s}{\mathbb{S}}
\newcommand{\V}{\mathbb{V}}
\newcommand{\T}{\mathbb{T}}
\newcommand{\U}{\mathbb{U}}

\newcommand{\Z}{\mathbb{Z}}
\newcommand{\Hom}{\textrm{Hom}}
\newcommand{\M}{\mathbb{M}}

\newcommand{\m}{\mathfrak{m}}

\numberwithin{equation}{section}

\begin{document}

\title[Subgroups of $GL_n$ over complete local Noetherian rings]{A structure theorem for subgroups of $GL_n$ over complete local Noetherian rings with large residual image}
\author{Jayanta Manoharmayum}
\address{School of Mathematics and Statistics, University of
    Sheffield,
    Sheffield S3 7RH,
    U.K.}
\email{J.Manoharmayum@sheffield.ac.uk}
\date{}

\begin{abstract}

Given a complete local Noetherian ring $(A,\m_A)$  with
 finite residue field and a subfield $\pmb{k}$ of $A/\m_A$, we show that every closed
 subgroup  $G$ of $GL_n(A)$ such that $G\mod{\m_A}\supseteq SL_n(\pmb{k})$  contains a conjugate of $SL_n(W(\pmb{k})_A)$ under some small restrictions on $\pmb{k}$. Here $W(\pmb{k})_A$ is the
 closed subring of $A$ generated by the Teichm\"{u}ller lifts of elements of the subfield $\pmb{k}$.

\end{abstract}

\maketitle

\section{Introduction}

Let $\pmb{k}$ be a finite field of characteristic $p$ and let $W(\pmb{k})$ be its Witt ring.
Then, by the structure theorem for complete local rings (see Theorem 29.2 of
\cite{matsumura}), every complete local ring
with residue field containing $\pmb{k}$ is naturally a $W$-algebra. More precisely,
given a complete local ring $(A,\m_A)$ with maximal ideal $\m_A$ and a field homomorphism $\overline\phi:\pmb{k}\to A/\m_A$, there is a unique homomorphism
$\phi:W(\pmb{k})\to A$ of
local rings which induces $\overline\phi$ on residue fields. The homomorphism $\phi$ is
completely determined by its action on Teichm\"{u}ller lifts:
if $x\in \pmb{k}$ and $\hat{x}\in W(\pmb{k})$ is its Teichm\"{u}ller then $\phi(\hat{x})$ is the
Teichm\"{u}ller lift of $\overline\phi(x)$.

In this article, we consider an
`\emph{analogous}' property for subgroups of $GL_n$ over complete local Noetherian rings. From here on the index $n$ is fixed and assumed to be at least $2$. First a small bit of notation before we state our result formally: Given  a complete local ring  $(A,\m_A)$ and a finite subfield $\pmb{k}$  of
the residue field $A/\m_A$, denote by $W(\pmb{k})_A$ the image of the natural local homomorphism $W(\pmb{k})\to A$ from
the structure theorem. Alternatively, $W(\pmb{k})_A$ is the smallest closed subring of $A$ containing the Teichm\"{u}ller lifts of $\pmb{k}$.

\begin{mainthm} Let $(A,\m_A)$ be a complete local Noetherian ring with maximal ideal
$\m_A$ and finite residue field $A/\m_A$ of characteristic $p$. Suppose we are given a subfield $\pmb{k}$ of $A/\m_A$ and a closed
 subgroup  $G$ of $GL_n(A)$. Assume that:
 \begin{itemize}
 \item The cardinality of $\pmb{k}$ is at least $4$. Furthermore, assume that
 $\pmb{k}\neq\F_5$ if $n=2$ and that $\pmb{k}\neq \F_4$ if $n=3$.
  \item $G\mod{\m_A}\supseteq SL_n(\pmb{k})$.
      \end{itemize}
      Then $G$ contains a conjugate of $SL_n(W(\pmb{k})_A)$.
\end{mainthm}

For an application, set $W_m:=W(\pmb{k})/p^m$ and $G:=SL_n(W_m)$ with
$\pmb{k}$  as in the above theorem. Then the above result implies that  $W_m$, with the natural representation $\rho:G\to SL_n(W_m)$, is the universal deformation ring for deformations of
$\overline\rho:=\rho\mod{p}: G\to SL_n(\pmb{k})$
in the category of complete local Noetherian rings with residue field $\pmb{k}$.
(See Remark \ref{application}.)

We now outline the structure of this article (and introduce some notation along the way). If $M$ is
a module over a commutative ring $A$, then $\M(M)$, resp. $\M_0(M)$, denotes the $GL_n(A)$-module
of $n$ by $n$ matrices over $M$, resp.  $n$ by $n$ trace $0$ matrices over $M$,
with $GL_n(A)$ action given by conjugation. The bi-module structure on $M$ is of course given by $amb:=abm$
for all $a,b\in A$, $m\in M$. A typical application of this consideration is when $B=A/J$ for some
ideal $J$ with $J^2=0$. Then $GL_n(B)$ acts on $\M(J)$ and $\M_0(J)$, and this action is compatible
with the action of $GL_n(A)$.

Given $A$, $B$ and $J$ as above, we can understand subgroups of $SL_n(A)$ if we know
enough  about extensions of $SL_n(B)$ by $\M_0(J)$. We give a brief description of the process involved (in terms of group extensions) in section 2. Determining extensions in general can be a complicated problem
but, for the proof of the main theorem, we only need to look at
extensions of $SL_n(W(\pmb{k})/p^m)$ by $ \M_0(\pmb{k})$. To carry out the argument we need some
control over $H^1(SL_n(W(\pmb{k})/p^m), \M_0(\pmb{k}))$ and $H^2(SL_n(W(\pmb{k})/p^m), \M_0(\pmb{k}))$.  Some care is needed when  $p$ divides $n$; the necessary calculations are carried out in section 3.

We remark that the condition on the residual image of $G$ is necessary for the
calculations used here to work. There are results  due to Pink (see \cite{pink}) characterising closed  subgroups of $SL_2(A)$ when the complete local ring $A$ has
odd residue characteristic. (The proof depends on  matrix/Lie algebra identities that only work when $n=2$.) For explicit descriptions of some classes of subgroups of $SL_2(A)$, see B\"{o}ckle \cite{boeckle}.

A different aspect of the size of closed subgroups of $GL_n(A)$ with large residual image is studied by Boston in \cite{mazur wiles}.  In a sense our result complements that of Boston: we give a lower bound for the
size of closed subgroups assuming the image modulo $\m_A$ is big enough, while Boston's result there, \emph{loc. cit}, says such subgroups will contain $SL_n(A)$ if the image modulo $\m_A^2$ is big enough.

\section{Twisted semi-direct products}
Let  $G$ be a finite group. Given an $\F_p[G]$-module $V$ and a normalised $2$-cocyle $x:G\times G\to V$,  we can then form the \emph{twisted semi-direct product} $V\rtimes_x G$. Here, normalised means that $x(g,e)=x(e,g)=0$ for all $g\in G$ where we have denoted the identity of $G$ by $e$. Recall  $V\rtimes_xG$ has elements $(v,g)$ with $v\in V$, $g\in G$ and composition
\[
(v_1,g_1)(v_2,g_2):= (x(g_1,g_2)+v_1+g_1v_2, g_1g_2),
\] and that the cohomology class of $x$ in $H^2(G,V)$ represents the extension
\begin{equation}\label{twisted semi-direct}
0\xrightarrow{}V\xrightarrow{v\to (v,e)}V\rtimes_xG \xrightarrow{(v,g)\to g}G\xrightarrow{}e.
\end{equation}
The conjugation action of  $V\rtimes_x G$ on  $V$ is the one given by the $G$ action on $V$ i.e.  $(u,g)v:=(u,g)(v,e)(u,g)^{-1}=(gv,e)$ holds for all $u, v\in V$, $g\in G$.

We record the following result for use in the next section.
\begin{proposition}\label{image of transgression}
With $G$, $V$ and $x:G\times G\to V$ as above, let $\phi:V\rtimes_xe\to V$ be the map $(v,e)\to -v$. Then under the transgression map
\[ \delta :\Hom_G(V\rtimes_x e ,V)= H^1(V\rtimes_xe , V)^G\to H^2(G,V),\]
$\delta(\phi)$ is the class of $x$.
\end{proposition}
\begin{proof} Let $\pi: V\rtimes_xG\to V$ be the map given by $\pi(v,g):=-v$.  Thus $\pi|_{V\rtimes_xe}=\phi$ and
$\pi(ab)=\pi(a)+a\pi(b)a^{-1}$ whenever $a$ or $b$ is in $V\rtimes_x e$.  The map
$\partial\pi :G\times G\to V$ given by
 $\partial \pi (g_1,g_2):= \pi(a_1)+a_1\pi(a_2)a_1^{-1}-\pi(a_1a_2)$ where $a_i\in V\rtimes_x G$ lifts $g_i$ is then well defined  and $\delta(\phi)$ is the class of $\partial \pi$.
(See Proposition 1.6.5
in  \cite{NSW}.) Taking  $a_i:=(0,g_i)$ we see that $\partial \pi (g_1,g_2)=x(g_1,g_2)$.
\end{proof}

 For the remainder of this section, we assume that we are given  an
  $\F_p[G]$-module $M$ of finite cardinality and an  $\F_p[G]$-submodule $N\subseteq M$ such that the map
 \begin{equation}\label{injective}
H^2(G,N)\to H^2(G,M)\quad\text{  is injective,}
 \end{equation}  and fix a
  normalised $2$-cocycle  $x: G\times G\to N$.  As we shall see,  assumption \ref{injective} pretty much determines  $N\rtimes_x G$ as a subgroup of $M\rtimes_xG$  up to conjugacy.

Suppose we are given a subgroup $H$  of $M\rtimes_x G$ extending $G$ by $N$ i.e. the sequence
\begin{equation}\label{exact2}\begin{CD}
0@>>>  N  @>>> H@>{(m,g)\rightarrow g}>> G @>>> e\end{CD}
 \end{equation}
 is exact.
 By  assumption \ref{injective}, the extension  \ref{exact2}  must
 correspond to $x$ in $H^2(G, N)$.  Hence
 there is an isomomorphism $\theta:N\rtimes_x G\to H$ such that the diagram
\begin{equation}\label{defn theta}\begin{CD}
0  @>>> N @>>> N\rtimes_x G  @>>> G @>>>e\\
@.      @|   @VV{\theta}V       @|     @. \\
0  @>>> N @>>> H @>>> G @>>> e
\end{CD}\end{equation}
commutes, and this allows us to
 define a map $\xi :G\to M$ so that the relation $\theta (0,g)=(\xi(g),g)$ holds for all $g\in G$.

 \begin{proposition}\label{coh prop}
 With notation and assumptions as above, we have:
 \begin{enumerate}[(i)]
 \item $\theta(n,g)=(n+\xi(g), g)$ for all $n\in N$, $g\in G$.
 \item  The map $\xi: G\to M$ is a $1$-cocycle.
 \item If $H^1(G,M)=0$ then $\theta$ is  conjugation by $(m,e)$ for some $m\in M$.
 \end{enumerate}
 \end{proposition}

 \begin{proof}
 \begin{enumerate}[(i)]
 \item This is a simple computation using the relation $(n,g)=(n,e)(0,g)$.
 \item
 Let $g_1,g_2\in G$. Using part (i), we get
  \begin{align*}
 \theta((0,g_1)(0,g_2))&=\theta(x(g_1,g_2),g_1g_2)=(x(g_1,g_2)+\xi(g_1g_2), g_1g_2),\quad\text{and}\\
 \theta((0,g_1)(0,g_2))&=(\xi(g_1),g_1)(\xi(g_2),g_2)
 =(x(g_1,g_2)+\xi(g_1)+g_1\xi(g_2),g_1g_2).
\end{align*}
Therefore we must have $\xi(g_1g_2)=\xi(g_1)+g_1\xi(g_2)$.

\item If $H^1(G,M)=0$ then  there exists an $m\in M$ such that $\xi(g)=gm-m$ for all $g\in G$. One then uses part (i) to check that
\[ (m,e)^{-1}(n,g)(m,e)=(n+gm-m,g)=\theta(n,g). \qedhere \]
 \end{enumerate}
 \end{proof}

We now give---with a view to motivating the calculations in the next section---a sketch
of how we  use the above proposition to prove a particular case of the main theorem. Suppose that we
have an Artinian local ring $(A,\m_A)$ with residue field $\pmb{k}$, and suppose that
we are given a subgroup $G\leq SL_n(A)$ with $G\mod\m_A=SL_n(\pmb{k})$. We'd like to know if a conjugate of $G$ contains $SL_n(W(\pmb{k})_A)$.

Suppose that $J$ is an ideal of $A$ killed by $\m_A$. To simplify the discussion further, let's assume that the quotient $A/J$ is $W_m:=W(\pmb{k})/p^m$, that
 $W(\pmb{k})_A=W(\pmb{k})/p^{m+1}$, and  that
$G\mod{J}=SL_n(W_m)$. The assumption that $W(\pmb{k})_A=W(\pmb{k})/p^{m+1}$ gives us a choice $\pmb{k}\subseteq J$, and
we can set up an identification of $SL_n(A)$ with a twisted semi-direct product
$\M_0(J)\rtimes_xSL_n(W_m)$ so that the  subgroup $SL_n(W(\pmb{k})_A)$ gets identified with
$\M_0(\pmb{k})\rtimes_xSL_n(W_m)$.
 In order to apply Proposition \ref{coh prop} and conclude that $G$ is, up to conjugation, $M\rtimes_xSL_n(W_m)$ for some $\F_p[SL_n(W_m)]$-submodule $M$ of $\M_0(J)$, we need to verify that:
\begin{itemize}
\item Assumption
\ref{injective} holds for $\F_p[SL_n(W_m)]$-submodules of $\M_0(J)$ (Theorem \ref{injective H2});
\item $H^1(SL_n(W_m),\M_0(J))=(0)$. This is a consequence of known results when $m=1$ (Theorem \ref{coh cps}) and Proposition \ref{coh sln} in `good' cases. Extra arguments
(cf, for instance, Proposition \ref{coh sln p|n}) are needed when $p$ divides $n$.
\end{itemize}
We can then conclude that a conjugate of $G$ contains $SL_n(W(\pmb{k})_A)$ provided
$\M_0(\pmb{k})\subset M$. This is derived from the injectivity of $H^2$s (in particular Corollary \ref{main corollary}); see  claim \ref{claim 1} in section 4.

\section{Cohomology of $SL_n(W/p^m)$}
We fix, as usual,  a finite field $\pmb{k}$ of characteristic $p$ and set $W_m:=W/p^m$
where $W:=W(\pmb{k})$ is  the  Witt ring of $\pmb{k}$. From here on we assume $n\geq 2$.
Our aim is to verify that assumption \ref{injective} holds. More precisely, we have the following:

 \begin{theorem}\label{injective H2}
  Let $\pmb{k}$ be a finite field of characteristic $p$ and cardinality at least $4$. Suppose $N\subseteq M$ are $\F_p[SL_n(W_m)]$-submodules of $\M_0(\pmb{k})^r$ for some integer $r\geq 1$.
 Then the induced map on second cohomology $H^2(SL_n(W_m),N)\to
 H^2(SL_n(W_m),M)$ is injective.
 \end{theorem}

The proof of Theorem \ref{injective H2} relies on knowledge of the first cohomology  of $SL_n(W_m)$ with coefficients in $\M_0(\pmb{k})$.  There are a couple more $SL_n(W_m)$ modules to consider when $p$ divides $n$, and we introduce these: Write $\s$  for the subspace of scalar matrices in $\M_0(\pmb{k})$. Thus $\s=(0)$ unless $p$ divides $n$ in which case $\s=\{\lambda I : \lambda\in\pmb{k}\}$. If $p|n$ we define $\V:=\M_0(\pmb{k})/\s$.

The first cohomology of $SL_n(W_m)$ with coefficients in $\M_0(\pmb{k})$ or $\V$ is well understood when $m=1$, and  we refer to Cline, Parshall and Scott \cite[Table 4.5]{CPS} for the following the  following result. (For results on $H^2(SL_n(\pmb{k}),\M_0(\pmb{k}))$ see \cite{burichenko}, \cite{georgia alg group}.)
\begin{theorem}\label{coh cps} Assume that the cardinality of $\pmb{k}$ is at least $4$.
\begin{itemize}
\item  Suppose $(n,p)=1$. Then  $H^1(SL_n(\pmb{k}),\M_0(\pmb{k}))$ is always $0$ except for
$H^1(SL_2(\F_5),\M_0(\pmb{k}))$ which is a $1$-dimensional $\pmb{k}$-vector space.
\item Suppose $p|n$. Then $H^1(SL_n(\pmb{k}),\V)$ is  a $1$-dimensional $\pmb{k}$-vector space.
\end{itemize}
\end{theorem}

Throughout this section, we will denote by $\Gamma$ the kernel of the $\mod{p^m}$-reduction map
$SL_n(W_{m+1})\to SL_n(W_m)$. We have suppressed the dependence on $m$ in our notation; this shoudn't create any great inconvenience. If $M\in \M_0(W)$ is a trace $0$, $n\times n$-matrix with coefficients in
$W$ then $I+p^mM\mod{p^{m+1}}$ is in $\Gamma$, and this sets up a natural identification of
$\M_0(\pmb{k})$ and $\Gamma$ compatible with $SL_n(W_m)$-action. The extension of Theorem \ref{coh cps} to the group $SL_n(W_m)$ for arbitrary $m$, carried out  in subsections \ref{section H^1(M_0(k)} and \ref{section H^1(V)}, then relies on the injectivity of transgression maps from $H^1(\Gamma, -)^{SL_n(W_m)}$ to $H^2(SL_n(W_m),-)$.

We end---before we go into the main computations of this section---by reviewing the structure of $\M_0(\pmb{k})$, and therefore of $\Gamma$, as an $\F_p[SL_n(\pmb{k})]$-module. For $1\leq i,j\leq n$, $e_{ij}$ denotes the matrix unit which is $0$ at all places except at the $(i,j)$-th place where it is $1$.

\begin{lemma}\label{lemma1} Assume that $\pmb{k}\neq \F_2$ if $n=2$.
\begin{enumerate}[(i)]
\item If $X$ is an $\F_p[SL_n(\pmb{k})]$-submodule of $\M_0(\pmb{k})$ then either $X$ is a  subspace of $\s$, or $X=\M_0(\pmb{k})$. Thus $\M_0(\pmb{k})/\s$ is a simple $\F_p[SL_n(\pmb{k})]$-module, and the sequence
        \begin{equation}\label{exact sequence for ad}
0\xrightarrow{}\s\xrightarrow{}\M_0(\pmb{k})\xrightarrow{} \V
\xrightarrow{}0
\end{equation}
is non-split when $p|n$.
    \item If $\phi: \M_0(\pmb{k})\to \M_0(\pmb{k})$ is a homomorphism of $\F_p[SL_n(\pmb{k})]$-modules then there exists a $\lambda\in \pmb{k}$ such that $\phi(A)=\lambda A$ for all $A\in \M_0(\pmb{k})$.
    \item Suppose $p|n$ and $\phi: \M_0(\pmb{k})\to \V$ is a homomorphism of $\F_p[SL_n(\pmb{k})]$-modules. Then $\phi(\s)=(0)$ and the induced map $\phi:\V\to\V$ is multiplication by a scalar in $\pmb{k}$.
\end{enumerate}
\end{lemma}

\begin{proof} Let $\U$ be the subgroup $SL_n(\pmb{k})$ consisting of upper triangular matrices with
ones on the diagonal. As an $\F_p[\U]$-module the semi-simplification of $\M_0(\pmb{k})$ is a direct sum of copies of $\F_p$ and $\M_0(\pmb{k})^\U=\s+\pmb{k}e_{1n}$.  Therefore if the $\F_p[SL_n(\pmb{k})]$-submodule $X$ is not a subspace of $\s$ then $X$ contains a matrix  $aI+be_{1n}$ with $b\neq 0$.

Suppose first that $a=0$.  By considering the action of diagonal matrices, we see that $X$ must in fact contain the full $\pmb{k}$-span of $e_{1n}$. Conjugation by $SL_n(\pmb{k})$ then implies that $X\supseteq \pmb{k}e_{ij}$ whenever $i\neq j$. Now,  under the action of $SL_n(\pmb{k})$, we can conjugate $e_{ij}+e_{ji}$ with  $i\neq j$ to $e_{ii}-e_{jj}$ when $p$ is odd and to
$e_{ii}-e_{jj}+e_{ij}$ when $p=2$. In any case, we can conclude that $X\supseteq \pmb{k}(e_{ii}-e_{jj})$ whenever $i\neq j$. It follows that $X$ must be the whole space $\M_0(\pmb{k})$.

Suppose now $a\neq 0$. Thus $\s\neq 0$ and $p$ divides $n$. When $n\geq 3$ the relation
\[
(I+e_{21})(aI+be_{1n})(I-e_{21})=aI+be_{1n}+be_{2n}\]
implies $be_{2n}$ and, consequently, $be_{1n}$ are in $X$, and so $X=\M_0(\pmb{k})$. When $n=2$---so $p=2$ and $\pmb{k}$ has at least $4$ elements---we can find a $0\neq \lambda \in \pmb{k}$ with $\lambda^2\neq 1$. Conjugating by
$\tiny\begin{pmatrix}\lambda &0\\ 0&1/\lambda\end{pmatrix}$, we see that $aI+b\lambda^2e_{1n}\in X$. This gives $0\neq b(\lambda^2-1)e_{1n}\in X$ and so $X=\M_0(\pmb{k})$.

Now for part (ii). Since $\phi$ commutes with the action of $SL_n(\pmb{k})$, the subspaces $\M_0(\pmb{k})^{SL_n(\pmb{k})}$ and
$\M_0(\pmb{k})^{\U}$ are invariant under $\phi$.
When  $p$  divides $n$ the first of these  gives $\phi\s\subseteq \s$; if $p$ does not divide $n$, then $\M_0(\pmb{k})^\U=\pmb{k}e_{1n}$ and so we must have $\phi(e_{1n})=\lambda e_{1n}$ for some $\lambda\in \pmb{k}$. In any case, we can find a $\lambda\in \pmb{k}$ such that the $\F_p[SL_n(\pmb{k})]$-module homomorphism $\phi-[\lambda]: \M_0(\pmb{k})\to \M_0(\pmb{k})$ given by $A\to \phi(A)-\lambda A$ has non-trivial kernel. We can then conclude,  by part (i) and a simple dimension count, that the kernel has to be the whole space $\M_0(\pmb{k})$, and therefore $\phi$ must be multiplication by $\lambda$.

For part (iii), that $\s\subseteq\ker\phi$ follows from part (i). The second part is proved along the same lines as the proof of part (ii) by considering $\phi(e_{1n})$.
\end{proof}

\subsection{Determination of $H^1(SL_n(W_m),\pmb{k})$}\label{section H^1(k)}
Let $\pmb{k}$ have cardinality $p^d$. Our aim is to show that $H^1(SL_n(W_m),\pmb{k})$
vanishes, subject to some mild restrictions on $\pmb{k}$. We do this inductively using
inflation--restriction after dealing with the base case $m=1$ by adapting Quillen's
result
in the general linear group case (see section 11 of \cite{quillen}).

To start off we impose no restrictions other than $n\geq 2$.
Denote
by $\T$ the subgroup of diagonal matrices in $SL_n(\pmb{k})$ and write $(t_1, t\ldots , t_n)$ for  the diagonal matrix with
$(i,i)$-th entry $t_i$. The image of the  homomorphism $\T\to  (\pmb{k}^\times)^{n-1}$
given by
\[
(t_1,\ldots ,t_n)\to (t_2/t_1,\ldots , t_n/t_{n-1})
\] has index $h:=\text{hcf} (n,p^d-1)$ in $(\pmb{k}^\times)^{n-1}$.
Taking this into account and following the remark at end of section 11 of
\cite{quillen}, the proof covering the general
linear group case  only needs a small modification at one place\footnote{The congruence
just before Lemma 16 changes to a congruence modulo $(p^d-1)/h$.} to give the following:

\begin{theorem} \label{quillen sln} Let $\pmb{k}$ be a finite field of characteristic $p$ and cardinality
$p^d$. Then $H^i(SL_n(\pmb{k}),\F_p)=0$ for $0<i<d(p-1)/h$ where $h:=\text{hcf}(n,p^d-1)$.
\end{theorem}

For a fixed $n$, Theorem \ref{quillen sln} implies the vanishing of
$H^1(SL_n(\pmb{k}),\pmb{k})$ and  $H^2(SL_n(\pmb{k}),\pmb{k})$ for fields with
sufficiently large cardinality. To get a stronger result for $H^1$ and $H^2$ covering fields with small cardinality, we will need to carry out  a slightly more detailed analysis.

In order to show $H^\ast(SL_n(\pmb{k}),\F_p)=0$ it is enough
to check that $H^\ast(\U,\F_p)^{\T}=0$ where $\U$ is the subgroup of upper triangular matrices with
ones on the diagonal. Fix an algebraic closure $\overline\F_p$ of $\F_p$ containing $\pmb{k}$.  Since $\T$ is an abelian group of order prime to $p$, the
$\overline\F_p[\T]$-module $H^\ast(\U,\F_p)\otimes_{\F_p}\overline\F_p$ is
isomorphic to a direct sum of characters; we will then have to check that none of these  can be
the trivial character.

Let $\Delta^+$ be the set of characters
$a_{ij}:\T\to \pmb{k}^\times$ given by
$a_{ij}(t_1,\ldots ,t_n):=t_i/t_j$ where $1\leq i< j\leq n$. The analysis in \cite[section 11]{quillen} shows that the Poincar\'{e} series
of $H^*(\U)$ as a representation of $\T$, denoted by $\text{P.S.}(H^*(\U))$, satisfies
the bound
\begin{equation}\label{upper bound 1}
\text{P.S.}(H^*(\U)):=
\sum_{i\geq 0}\text{cl}(H^i(\U,\F_p)\otimes_{\F_p}\overline\F_p)z^i \ll
\prod_{a\in\Delta^+ }\prod_{b=0}^{d-1}\frac{1+a^{-p^b}z}{1-a^{-p^b}z^2}
\end{equation}
in $R_{\overline\F_p}(\T)[[z]]$. Here $R_{\overline\F_p}(\T)$ is the
Grothendieck group for representations of $\T$ over $\overline\F_p$, and $\text{cl}(V)$
is the class of a $\overline\F_p[\T]$-module $V$
in $R_{\overline\F_p}(\T)$; given $\overline\F_p[\T]$-modules $V_0, V_1, V_2,\ldots$ and $W_0, W_1, W_2,\ldots$, the bound
\[
\sum_{i\geq 0}\text{cl}(W_i)z^i\ll \sum_{i\geq 0}\text{cl}(V_i)z^i \]
in $R_{\overline\F_p}(\T)[[z]]$ expresses the property that  $W_i$ is isomorphic to an $\overline\F_p[\T]$-submodule of $V_i$ for every integer $i\geq 0$. Thus the right hand side of \ref{upper bound 1} tells us which characters \emph{might} occur in the decomposition of the
$\overline\F_p[\T]$-module $H^\ast(\U,\F_p)\otimes_{\F_p}\overline\F_p$.

 Note that our choice of a positive root system $\Delta^+$ is different from the one  in \cite{quillen}; the choice
made there leads to a sign discrepancy in the upper
bound \ref{upper bound 1} (but doesn't affect any of the results derived from it).  If we use the ordering on
$\Delta^+$ given by  $(i',j')\leq (i,j)$
if either $i'<i$, or $i'=i$ and $j\leq j'$, then with  notation as in \cite{quillen}  we have a central extension
\[ 0\to \pmb{k}_a\to\U/\U_a\to \U/\U_{a'}\to 1\] with $\T$-action and the argument in
\cite{quillen}
carries
through verbatim.

It is then
straightforward to work out
the coefficients of $z$ and $z^2$ on the right hand side of \ref{upper bound 1}, and we
can conclude the following: If  $\chi: \T\to \overline\F_p^\times$ is a character
occurring in $\text{cl}(H^i(\U,\F_p)\otimes_{\F_p}\overline\F_p)$, $i=1, 2$, then
$\chi^{-1}$ is either
\begin{itemize}
\item a Galois conjugate of a positive root i.e.  $\chi^{-1}=a^{p^b}$  for some positive root $a\in\Delta$ and integer $0\leq b< d$, or
\item  a  product  $\alpha\alpha'$ where $\alpha$, $\alpha'$ are Galois conjugates
of positive roots and $\alpha\neq \alpha'$. (This case  happens only when $i=2$.)
\end{itemize}

Thus, taking Galois
conjugates as needed, we need to determine when  $a_{ij}$ or
 $a_{ij}a_{kl}^{p^b}$ is the trivial character, where $a_{ij}, a_{kl}\in \Delta^+$ and $0<b<d$ in the case
 $(i,j)=(k,l)$. The first case is immediate: $a_{ij}$ is never the trivial character except when $\pmb{k}=\F_2$, or $n=2$ and
$\pmb{k}=\F_3$.

Now for the second case. We now have  integers $1\leq i<j\leq n$, $1\leq k<l\leq n$,
$0\leq b<d$ with $b\neq 0$ if $(i,j)=(k,l)$ such that the following relation holds:
\begin{equation}\label{relation 1}
\frac{t_i}{t_j}\Big(\frac{t_k}{t_l}\Big)^{p^b}=1\qquad \text{for all}\quad
(t_1,\ldots , t_n)\in \T.
\end{equation}
We will determine for which fields the above relation holds
by specialising suitably. We exclude $\pmb{k}=\F_2$ in what follows.

Firstly, let's consider the case when $i$, $j$, $k$ and $l$ are distinct. Thus $n\geq 4$. We can specialise \ref{relation 1} to $t_k=t_l=1$ and
$t_i=t_j^{-1}=t$ for $t\in \pmb{k}^\times$. We then get $t^2=1$ for all
 $t\in \pmb{k}^\times$---which implies $\pmb{k}$ can only be $\F_3$. Furthermore, if $n\geq 5$ we have an even better specialisation: we can choose $t_j=t_k=t_l=1$ and $t_i$ freely, and  conclude \ref{relation 1} never holds.

Next, suppose the cardinality $\{i,j,k,l\}$ is $3$. If we suppose $\{i,j,k,l\}=\{i,k,l\}$ (the case $\{i,j,k,l\}=\{j,k,l\}$ is similar), then specialisation to $t_j=t_k=t_l=t^{-1}$ and
$t_i=t^2$ implies that $t^3=1$ for all $t\in \pmb{k}^\times$ i.e.  $\pmb{k}$ is a subfield of $\F_4$. If in addition $n\geq 4$ we can take $t_k=t_l=1$ and then there is
a free choice for either $t_i$, so \ref{relation 1} cannot hold.

Finally consider the case when the cardinality of $\{i,j,k,l\}$ is $2$. We must then
have $i=k$, $j=l$ and $1\leq b<d$. Taking $t_i=t=t_j^{-1}$, we get
$t^{2(1+p^b)}=1 $ for all $t\in \pmb{k}^\times$, and so $2(1+p^b)=p^d-1$. This only works when $\pmb{k}=\F_9$. Moreover,
when $n\geq 3$, we can set $t_j=1$ and then the relation \ref{relation 1} implies $t^{p^b+1}=1$ for all $t\in \pmb{k}^\times$. So $p^b+1=p^d-1$ and $\pmb{k}$ is necessarily $\F_4$. Therefore in the case $(i,j)=(k,l)$ the relation
\ref{relation 1} holds only when $n=2$ and $\pmb{k}=\F_9$.

We have thus proved the first part of the following:
\begin{theorem}\label{trivial H1 H2}
Let $\pmb{k}\neq \F_2$ be a finite field of characteristic $p$ and let $n\geq 2$ be an integer. Further,  assume that
\begin{itemize}
\item if $n=4$ then $\pmb{k}$ is not $\F_3$;
\item if $n=3$ then $\pmb{k}\neq \F_4$;
\item if $n=2$ then $\pmb{k}$ is not $\F_3$ or $\F_9$.
\end{itemize} Then
$H^1(SL_n(\pmb{k}),\F_p)$ and $H^2(SL_n(\pmb{k}),\F_p)$ are both trivial. Furthermore, under the same assumptions on $\pmb{k}$, we have $H^1(SL_n(W_m),\pmb{k})=(0)$ for all integers $m\geq  1$.
\end{theorem}
The second part is proved by induction using inflation-restriction and the vanishing of $H^1(SL_n(\pmb{k}),\pmb{k})$ from the first part. With $\Gamma=\ker(SL_n(W_{m+1})\to SL_n(W_m))$ we have
\[ 0\to H^1(SL_n(W_m),\pmb{k})\to H^1(SL_n(W_{m+1}),\pmb{k})\to H^1(\Gamma,\pmb{k})^{SL_n(W_m)}.\]
Now the natural identification of $\M_0(\pmb{k})$ with $\Gamma$ compatible with $SL_n(W_m)$-actions  sets up an isomorphism between
$H^1(\Gamma,\pmb{k})^{SL_n(W_m)}$ and $\Hom_{\F_p[SL_n(\pmb{k})]}(\M_0(\pmb{k}),\pmb{k})$.
The latter vector space is easily seen to be $(0)$ by a dimension count using Lemma \ref{lemma1}, and the theorem follows.

\subsection{Determination of $H^1(SL_n(W_m),\M_0(\pmb{k}))$} \label{section H^1(M_0(k)}
The result here is that all  cohomology classes come from $H^1(SL_n(\pmb{k}),\M_0(\pmb{k}))$. More precisely:
\begin{proposition}\label{coh sln}
Suppose that $\pmb{k}$ has cardinality at least $4$ and that  $\pmb{k}\neq \F_4$ when $n=3$. The inflation map
$H^1(SL_n(W_{m}),\M_0(\pmb{k})) \to H^1(SL_n(W_{m+1}),\M_0(\pmb{k}))$
is then an isomorphism for all integers $m\geq  1$.
\end{proposition}

By the inflation--restriction exact sequence, the above proposition follows if we can
show that the transgression map
\[ \delta:H^1(\Gamma,\M_0(\pmb{k}))^{SL_n(W_{m})}\to
 H^2(SL_n(W_{m}),\M_0(\pmb{k}))\]
 is injective. Since  $H^1(\Gamma,\M_0(\pmb{k}))^{SL_n(W_{m})}$ has dimension $1$ as a $\pmb{k}$-vector space by Lemma \ref{lemma1}, we just need to check that $\delta$ is not the zero map.

Recall that we have a natural identification of $\Gamma$ with $\M_0(\pmb{k})$ given by $\phi(I+p^mA):=A\mod{p}$.
Hence by Proposition \ref{image of transgression}, we see that $\delta(-\phi)$  must be the class of the extension
\[I\to\Gamma\to SL_n(W_{m+1})\to SL_n(W_m)\to I.\]
  Therefore the required conclusion follows if
 the above extension is non-split, and we address this below.

\begin{proposition} \label{no section}
Assume that $\pmb{k}$ has cardinality at least $4$ and that if $n=3$ then $\pmb{k}\neq \F_4$. Then the extension
  \begin{equation}\label{gamma}
I\xrightarrow{}\Gamma\xrightarrow{} SL_n(W_{m+1})\xrightarrow{} SL_n(W_m)
\xrightarrow{}I\end{equation}
does not split for any integer $m\geq  1$.
\end{proposition}
\begin{proof} This should be well known, but it is hard to find a reference in the form we need. We therefore sketch a proof for completeness. The case when $n=2$ and $p\geq 5$ is discussed  in
\cite{serre}.  For the non-splitting of the above sequence when $\pmb{k}=\F_p$ see
\cite{sah1}; for non-splitting in the $GL_n$ case see \cite{sah2}.

 If  $R$ is a commutative ring and $r\in R$ then we write $N(r)$ for  the elementary nilpotent $n\times n$ matrix in $\M(R)$ with zeroes in all places except at the $(1,2)$-th entry where it is $r$.
 Note that
 $N(r)^2=0$ and that
 \[
 (I+N(r))^k= I+kN(r)=I+krN(1)\]
 for every integer $k$.

Suppose there is a homomorphism $\theta: SL_n(W_m)\to SL_n(W_{m+1})$ which splits the
above exact sequence \ref{gamma}. We fix a section $s:W_m\to W_{m+1}$ that sends Teichm\"{u}ller lifts to Teichm\"{u}ller lifts. For instance, if we think in terms of Witt vectors of finite length then
we can take $s$ to be the map $(a_1,\ldots, a_m)\to (a_1,\ldots ,a_m,0)$. Finally, take
the map $A: W_m\to \M_0(\pmb{k})$   so that  the relation
\[
\theta(I+N(x))=(I+p^mA(x))(I+N(s(x)))\]
holds for all $x\in W_m$ (and we have abused notation and identified $p^mW_{m+1}$ with $p^m\pmb{k}$).

Now $\theta(I+N(x))$ has order dividing $p^m$ in $SL_n(W/p^{m+1})$ for any $x\in W_m$.
Writing $N$ and $A$ in lieu of $N(s(x))$ and $A(x)$, we have
\[
(I+N)^k(I+p^mA)(I+N)^{-k}=I+p^m(A+kNA-kAN-k^2NAN)
\]
for any integer $k$, and a small calculation yields
\begin{equation}\label{formula 1}
\theta(I+N(x))^{p^m}
=\big(I+\alpha p^m(NA-AN)-\beta p^mNAN\big)(I+p^mN).
\end{equation} where
$\alpha=p^m(p^m-1)/2$ and
$\beta=p^m(p^m-1)(2p^m-1)/6$.
Hence if either $p\geq 5$, or $p$ divides $6$ and $m\geq 2$, then
$\theta(I+N(1))$ cannot have order $p^m$---a contradiction.

From here on $p$ divides $6$ and $m=1$; so $\theta: SL_n(\pmb{k})\to SL_n(W/p^2)$ and
$s(x)=\hat{x}$.  Applying $\theta$ to $(I+N(x))(I+N(y))=I+N(x+y)$ and multiplying by $N(1)$ on the left and right then
gives $N(1)A(x)N(1)+N(1)A(y)N(1)=N(1)A(x+y)N(1)$, and therefore
\[
a_{21}(x+y)=a_{21}(x)+a_{21}(y)\]
 for all $x,y\in \pmb{k}$.

Suppose now $p=3$. The expression \ref{formula 1} for $\theta(I+N(x))^p$ then becomes
\[ I+pxN(1)+px^2N(1)A(x)N(1)=I.\] Comparing the $(1,2)$-th entries on both sides we get $x^2a_{21}(x)+x=0$ for all $x\in \pmb{k}$. Thus for $x\neq 0$ we have
$a_{21}(x)=-x^{-1}$. This contradicts the linearity of $a_{21}$ if $\pmb{k}\neq \F_3$.

Before we  consider the case $p=2$ specifically, we make some relevant simplifications
by considering  the action of $\T$, the  subgroup of diagonal
matrices in $SL_n(\pmb{k})$. For $t=(t_1,\ldots , t_n)\in SL_n(\pmb{k})$ we define
$\hat{t}:=(\widehat{t_1},\ldots ,\widehat{t_n})\in SL_n(W/p^2)$. We must then have
 $\theta(t)=B(t)\hat{t}$ where $B: \T\to \Gamma$ is a $1$-cocycle. Since $H^1(\T,\Gamma)=0$ we can assume, after conjugation by a matrix in $\Gamma$ if necessary,
 that $\theta(t)=\hat{t}$. The homomorphism condition applied to
 $\theta(t(I+N(x))t^{-1})$ then gives
 \[
 (I+pA(t_1x/t_2))(I+\widehat{(t_1x/t_2)}N(1))=(I+ptA(x)t^{-1})
 (I+\hat{t}\hat{x}N(1)\hat{t}^{-1})
 \]
where $t=(t_1,\ldots, t_n)$. Hence $A(t_1x/t_2)=tA(x)t^{-1}$ for all $t\in T$ and $x\in \pmb{k}$. By considering
specialisations $t_1=t_2=1$ for $n\geq 4$ and $t=(\lambda, \lambda, \lambda^{-2})$ when
$n=3$, we conclude that $a_{ij}(x)=0$ if $i\neq j$ and $i\geq 3$ or $j\geq 3$ provided
$\pmb{k}$ has cardinality at least $4$ and $\pmb{k}\neq \F_4$ when $n=3$.

We now go back to assuming $p=2$ and $m=1$. Relation \ref{formula 1} then becomes
\[
I+px(N(1)A(x)+A(x)N(1))+px^2N(1)A(x)N(1)+pxN(1)=I,\]
and we get $a_{21}(x)=0$ and $a_{11}(x)+a_{22}(x)=1$ whenever $x\neq 0$. Hence if
$\pmb{k}$ has cardinality at least $4$ and $\pmb{k}\neq \F_4$ when $n=3$, then
$\theta(I+N(x))$ is an upper-triangular matrix and so $a_{ii}(x+y)=a_{ii}(x)+a_{ii}(y)$
for $i=1,\ldots ,n$ and $x,y\in \pmb{k}$. Since $\pmb{k}$ has at least $4$ elements we
can choose $x,y\in \pmb{k}$ with $xy(x+y)\neq 0$, and this gives
\[
1=a_{11}(x+y)+a_{22}(x+y)=(a_{11}(x)+a_{11}(y))+(a_{22}(x)+a_{22}(y))=1+1\]
---a contradiction.
\end{proof}

\subsection{$H^1$ when $n$ and $p$ are not coprime} \label{section H^1(V)}

Suppose now that $p$ divides $n$. Thus $\M_0(\pmb{k})$ is reducible and we have the
exact sequence
 \begin{equation} \label{exact sequence p|n}
0\to \s\xrightarrow{i} \M_0(\pmb{k})\xrightarrow{\pi} \V\to 0.
\end{equation}
We then have the following analogue of Proposition \ref{coh sln}.
\begin{proposition} \label{coh sln p|n}
Assume that $p$ divides $n$ and that the cardinality of $\pmb{k}$ is at least $4$.
The inflation map
$H^1(SL_n(W_{m}),\V)\to H^1(SL_n(W_{m+1}),\V)$
is then an isomorphism for all integers $m\geq  1$.
\end{proposition}

Denote by $Z$ the subgroup of $\Gamma$ consisting of the scalar matrices
 $(1+p^m\lambda)I$. We then have an exact sequence
 \begin{equation}\label{gamma/Z}
 I\xrightarrow{}\Gamma/Z\xrightarrow{} SL_n(W_{m+1})/Z\xrightarrow{\mod{p^m}} SL_n(W_m)
\xrightarrow{}I.
\end{equation} Under the natural identification $\phi :\Gamma\to \M_0(\pmb{k})$  given by  $\phi(I+p^mA):=A\mod{p}$
 of  $\Gamma$ with $\M_0(\pmb{k})$, the groups $Z$, resp. $\Gamma/Z$, correspond to
 $\s$, resp. $\V$.
If we set $\psi:\Gamma/Z\to \V$ to be the map
 induced by $\phi\mod{\s}$, then Proposition
\ref{image of transgression} shows that  $\delta(-\psi)$ is the cohomology class of the extension  \ref{gamma/Z}  under the transgression map
 \[
 \delta : H^1(\Gamma/Z,\V)^{SL_n(W_m)}\to H^2(SL_n(W_m),\V).\]
Now, by Lemma \ref{lemma1}, the map \[ H^1(\Gamma/Z,\V)^{SL_n(W_{m})}\to H^1(\Gamma,\V)^{SL_n(W_{m})}\]
is an isomorphism of $1$-dimensional $\pmb{k}$-vector spaces. Thus the conclusion of
Proposition \ref{coh sln p|n} holds exactly when the extension \ref{gamma/Z} is
 non-split.

 In many cases the required non-splitting follows from a simple modification of the proof of Proposition \ref{no section}. More  precisely, we have the following:
\begin{lemma}\label{no section p|n}
Suppose $p|n$, and  assume that either $p\geq 5$ or $m\geq 2$. Then the extension
\[
 I\to\Gamma/Z\xrightarrow{} SL_n(W_{m+1})/Z\to SL_n(W_m)
\to I\]
does not split.
\end{lemma}
\begin{proof} We give a sketch: Suppose $\theta : SL_n(W_m)\to SL_n(W_{m+1})/Z$ is a
section. Then, with $N(1)$  the  elementary nilpotent matrix described in the
proof of Proposition \ref{no section}, we have $\theta(I+N(1))=(I+p^mA)(I+N(1))$ modulo $Z$ for some
$A\in \M_0(\pmb{k})$. Because elements in $Z$ are central,
relation \ref{formula 1} holds modulo $Z$ and the lemma easily follows.\end{proof}

We now deal with the  case $m=1$ and complete the proof of
Proposition \ref{coh sln p|n}. Consider the commutative diagram
 \begin{equation}\label{cd1}\begin{CD}
H^1(\Gamma,\M_0(\pmb{k}))^{SL_n(W_{m})}
@>{\delta}>>  H^2(SL_n(W_{m}),\M_0(\pmb{k})) \\
@VV{\pi^*}V   @VV{\pi^*}V        \\
H^1(\Gamma,\V)^{SL_n(W_{m})}
@>{\delta}>>  H^2(SL_n(W_{m}),\V)
\end{CD}\end{equation} where $\pi^*$ is the map induced by the projection
$\pi:\M_0(k)\to \V$. Now, the map $\pi^\ast$ on  the left hand side of the square
is an isomorphism by Lemma \ref{lemma1}. Since the cardinality of $\pmb{k}$ is at least $4$ (and remembering that  we are also assuming $p|n$),  the top row of the square
\ref{cd1} is an injection by
Proposition \ref{coh sln}. Furthermore,
Theorem \ref{trivial H1 H2} implies $H^2(SL_n(W_m),\pmb{k})=(0)$ and therefore
the map $\pi^\ast$ on the right hand side of the square
is an injection. Hence the bottom row of the square \ref{cd1} is also an injection and we can conclude the proposition.

\begin{remark}
As we saw in course of the proof, Proposition \ref{coh sln p|n} implies the following
extension of Lemma \ref{no section p|n}:

\begin{corollary}\label{non-split p|n}
Assume that $p$ divides $n$ and  $\pmb{k}$ has cardinality at least $4$. Then the sequence
\[
 I\to \Gamma/Z\to SL_n(W_{m+1})/Z\to SL_n(W_m)
\to I
\end{equation*}
does not split for any integer $m\geq  1$.
\end{corollary}
\end{remark}

We end this subsection with a description of the relations between the cohomology groups with coefficients $\M_0(\pmb{k})$, $\s$ and $\V$:
 \begin{proposition} \label{p dividing n} Suppose that $p$ divides $n$ and that $\pmb{k}$ has at least $4$ elements.  Then, with $i$ and  $\pi$ as in  the exact sequence \ref{exact sequence p|n},
the map $H^1(SL_n(W_m),\M_0(\pmb{k}) )\xrightarrow{\pi^*} H^1(SL_n(W_m),\V )$ is an isomorphism and
 \[ 0\to H^2(SL_n(W_m),\s )\xrightarrow{i^*} H^2(SL_n(W_m),\M_0(\pmb{k}) )\xrightarrow{\pi^*}
 H^2(SL_n(W_m),\V)\]
 is exact.
 \end{proposition}

 \begin{proof}
The long  exact sequence  obtained from \ref{exact sequence p|n}
shows  that we just need to check  $H^1(SL_n(W_m),\M_0(\pmb{k}) )\xrightarrow{\pi^*} H^1(SL_n(W_m),\V )$ is an isomorphism.
This holds when $m=1$ because both  $H^1(SL_n(\pmb{k}),\s)$ and $H^2(SL_n(\pmb{k}),\s)$ are $0$ by
Theorem \ref{trivial H1 H2}. For general $m$ we can use induction because in the commutative diagram
\begin{equation*}\begin{CD}
H^1(SL_n(W_{m}),\M_0(\pmb{k})) @>{\pi^*}>> H^1(SL_n(W_{m}),\V)\\
@VVV                          @VVV\\
H^1(SL_n(W_{m+1}),\M_0(\pmb{k})) @>{\pi^*}>> H^1(SL_n(W_{m+1}),\V)
\end{CD}\end{equation*}
the vertical inflation maps
 are isomorphisms by Proposition \ref{coh sln} and Proposition \ref{coh sln p|n}.
\end{proof}

\subsection{Proof of Theorem \ref{injective H2}} \label{section H^2}

Recall that we want to show the injectivity of $H^2(SL_n(W_m),N)\to
 H^2(SL_n(W_m),M)$  whenever $N\subseteq M$ are $\F_p[SL_n(W_m)]$-submodules of $\M_0(\pmb{k})^r$ for some integer $r\geq 1$.

 We will write $H^*(X)$ to mean $H^*(SL_n(W_m),X)$.  Note that it is enough to show that
 $H^2(M)\to
 H^2(\M_0(\pmb{k})^r)$ is injective for all $\F_p[SL_n(W_m)]$-submodules
 $M$ of $\M_0(\pmb{k})^r$. If $(n,p)=1$ then $\M_0(\pmb{k})^r$ is semi-simple and the desired
 injectivity is immediate. So we will  suppose $p$ divides $n$ from here on.

Consider the commutative diagram
 \begin{equation}\label{diagram 1}\begin{CD}
0  @>>> M\cap \s^r @>{i}>> M @>{\pi}>> M/(M\cap \s^r) @>>>0\\
@.      @VV{i}V   @VV{i}V       @VV{j}V     @. \\
0  @>>> \s^r@>{i}>> \M_0(\pmb{k})^r @>{\pi}>> \V^r @>>>0
\end{CD}\end{equation}
where the $i$'s are inclusions. Thus  $j$ is necessarily an injection.   Taking cohomology and using Proposition \ref{p dividing n}, we get a commutative diagram
\begin{equation}
\begin{CD}
@. H^2(M\cap \s^r)@>>> H^2(M) @>>> H^2(M/( M\cap \s^r))\\
@. @VV{\iota^*}V                             @VV{\iota^*}V                @VV{j^*}V\\
0@>>>H^2(\s^r)@>>> H^2(\M_0(\pmb{k})^r) @>>>
 H^2(\V^r)
\end{CD}
\end{equation}
in which the horizontal rows are exact. Now the maps $H^2(M\cap \s^r)\xrightarrow{i^*} H^2(\s^r)$ and
$H^2(M/( M\cap \s^r))\xrightarrow{j^*} H^2(\s^r)$ are injective since  $\s^r$
and $\V^r$ are semi-simple and $i$, $j$ are injections. A straightforward diagram chase
then shows that $i^*:H^2(M)\to H^2(\M_0(\pmb{k})^r)$ is an injection, and this
completes the proof of Theorem \ref{injective H2}.$\hfill\Box$

\medskip
As a consequence, we have the following:
 \begin{corollary} \label{main corollary} Let $\pmb{k}$ be a finite field of characteristic $p$ and cardinality at least $4$,  and let $M$, $N$ be two $\F_p[SL_n(W_m)]$-submodules of $\M_0(\pmb{k})^r$ for some integer $r\geq 1$. Suppose we are given
 $x\in H^2(SL_n(W_m), M)$ and $y\in H^2(SL_n(W_m),N)$ such that $x$ and $y$ represent the same cohomology class  in $H^2(SL_n(W_m),\M_0(\pmb{k})^r)$.
 Then there exists a $z\in H^2(SL_n(W_m), M\cap N)$ such that $x=z$, resp. $y=z$, holds in $H^2(SL_n(W_m), M)$, resp.
 $H^2(SL_n(W_m), M)$.
\end{corollary}
\begin{proof} Consider the exact sequence
\[
0\to M\cap N\xrightarrow{m\to m\oplus m}M\oplus N\xrightarrow{m\oplus n\to m-n} M+N\to 0.\]
By Theorem \ref{injective H2}, we get a short exact sequence
\[
0\to H^2(M\cap N)\to H^2(M)\oplus H^2(N)\to H^2(M+N).\]
Since $H^2(M+N)\to H^2(\M_0(\pmb{k})^r)$ is injective,  it follows that $x\oplus y$ is zero in $H^2(M+N)$ and therefore must be in the image of $H^2(M\cap N)$.
\end{proof}

\section{Proof of the main theorem}

From here on, we assume
that we are given finite fields $\pmb{k}\subseteq \pmb{k'}$
of characteristic $p$. Let $\mathcal{C}$ be the  category of complete local Noetherian
rings $(A,\m_A)$ with residue field $A/\m_A=\pmb{k'}$ and with morphisms required to be
identity on $\pmb{k'}$. We will abbreviate $W(\pmb{k})$ and $W(\pmb{k})_A$ for $A$ an
object in $\mathcal{C}$ to $W$ and $W_A$ respectively. Recall that  $W_A$ is the closed subring
of $A$ generated by the Teichm\"{u}ller lifts of elements of $\pmb{k}$; it is not an
object in $\mathcal{C}$ unless $\pmb{k}=\pmb{k'}$.  Throughout this section we assume that the finite field $\pmb{k}$ satisfies the hypothesis of the main theorem:
 \begin{assumption}\label{assumptions on k}
The cardinality of $\pmb{k}$ is at least $4$. Furthermore,
 $\pmb{k}\neq\F_5$ if $n=2$ and that $\pmb{k}\neq \F_4$ if $n=3$.
 \end{assumption}

Suppose we are given   a local ring $(A,\m_A)$  in $\mathcal{C}$ and
a closed subgroup $G$ of $GL_n(A)$ such
that $G\mod{\m_A}\supseteq SL_n(\pmb{k})$. We want to show that $G$ contains a
conjugate of $SL_n(W_A)$. Now, without any loss of generality, we may assume that
$G\mod{\m_A}= SL_n(\pmb{k})$. The quotient $G/(G\cap SL_n(A))$ is then pro-$p$. This
implies that $G\cap SL_n(A)\mod{\m_A}$ is a normal subgroup of $SL_n(\pmb{k})$ with
index a power of $p$. Now $PSL_n(\pmb{k})$ is simple since the cardinality of $\pmb{k}$
is at least $4$. Consequently we must have $G\cap SL_n(A)\mod{\m_A}=SL_n(\pmb{k})$.
Along with the fact that $A$ is the inductive limit of Artinian quotients $A/\m_A^n$,
we see that the main theorem follows from the following proposition:

\begin{proposition}\label{artinian case}
Let $\pi:(A,\m_A)\to (B,\m_B)$ be a surjection of Artinian local rings  in
 $\mathcal{C}$ with $\m_A\ker{\pi}=0$, and let $H$ be a subgroup of $SL_n(A)$ such that
 $\pi H=SL_n(W_B)$. Assume that $\pmb{k}$ satisfies assumption \ref{assumptions on k}. Then we can find a $u\in GL_n(A)$ such that $\pi u=I$ and
 $uHu^{-1}\supseteq SL_n(W_A)$.
 \end{proposition}

For the proof of the above proposition, let's set
$G:=\pi^{-1}SL_n(W_B)\cap SL_n(A)$ where  $\pi^{-1}SL_n(W_B)$ is the pre-image of
 $SL_n(W_B)$ under the map $\pi: GL_n(A)\to GL_n(B)$. We then have an exact sequence
\begin{equation}\label{seq 1}
0\xrightarrow{} \M(\ker{\pi}) \xrightarrow{j}\pi^{-1}SL_n(W_B)\xrightarrow{\pi}
SL_n(W_B)\rightarrow I
\end{equation}
 with $j(v)=I+v$ for $v\in \M(\ker\pi)$, and this restricts to
 \begin{equation}\label{seq 2}
0\xrightarrow{} \M_0(\ker{\pi}) \xrightarrow{j}G\xrightarrow{\pi}
SL_n(W_B)\rightarrow I.
\end{equation}
Note that $\M(\ker{\pi})\cong \M(\pmb{k})\otimes_{\pmb{k}}\ker \pi$ and $\M_0(\ker{\pi})\cong \M_0(\pmb{k})\otimes_{\pmb{k}}\ker \pi$ as $\pmb{k}\left[SL_n(W_B)\right]$-modules.

In what follows we will
abbreviate $H^\ast(SL_n(W_B), M)$ to simply $H^\ast(M)$.
For $X\subseteq SL_n(A)$, we set $\M_0(X)$ to be the set of matrices $v\in \M_0(\ker{\pi})$ such that
$j(v)\in X$. We then have the following::
\begin{claim}
\label{claim 1}
$\M_0(SL_n(W_A))\subseteq \M_0(H)$.
\end{claim}

Let's assume the above claim and carry on with the proof of Proposition \ref{artinian case}.
Fix a section $s:SL_n(W_B)\to SL_n(W_A)$ that sends identity to identity and set
$x:SL_n(W_B)\times SL_n(W_B)\to \M_0(SL_n(W_A))$ to be the resulting $2$-cocycle
representing the extension
\begin{equation}\label{exact seq for SL_n}
0\xrightarrow{}\M_0(SL_n(W_A))\xrightarrow{j} SL_n(W_A)\xrightarrow{} SL_n(W_B)
\xrightarrow{}I.\end{equation}
The section $s$ and cocyle $x$ thus set up an identification
\[
\varphi: \pi^{-1}SL_n(W_B)\to \M_0\rtimes_x SL_n(W_B),\]
and we have the following commutative diagram (cf diagram \ref{defn theta})
\begin{equation*}\begin{CD}
0  @>>> \M_0(H) @>>> \M_0(H)\rtimes_x SL_n(W_B) @>>> SL_n(W_B) @>>>I\\
@.      @|   @VV{\theta}V       @|     @. \\
0  @>>> \M_0(H) @>>> \varphi H @>>> SL_n(W_B) @>>>I.
\end{CD}\end{equation*}

Suppose first that $(p,n)=1$. Our assumptions on $\pmb{k}$ imply that we can combine Theorem \ref{coh cps} and Proposition \ref{coh sln} to conclude that $H^1(\M_0(\pmb{k}))=(0)$. Consequently, we get $H^1(\M_0(\ker\pi))=(0)$. Furthermore, $H^2(\M_0(H))\to H^2(\M_0(\ker\pi))$ is an injection by Theorem \ref{injective H2}. Hence we can
apply Proposition \ref{coh prop} and conclude that $\M_0(H)\rtimes_x SL_n(W_B)=\varphi u H u^{-1}$ for some $u\in G$ (cf. sequence \ref{seq 2}) with $\pi(u)=I$.

Suppose now $p$ divides $n$.  Since $H^1(\pmb{k})=0$ by Theorem \ref{trivial H1 H2}, we get the following
the exact sequence
\[0\to \pmb{k}\to H^1(\M_0(\pmb{k}))\to H^1(\M(\pmb{k}))\to 0\to H^2(\M_0(\pmb{k}))\to H^2(\M(\pmb{k}))\]
from $0\to \M_0(\pmb{k})\to \M(\pmb{k})\to \pmb{k}\to 0$. Now since $\dim_{\pmb{k}}H^1(\V)=1$ by Theorem \ref{coh cps} and Proposition \ref{coh sln p|n}, we must also have $\dim_{\pmb{k}}H^1(\M_0(\pmb{k}))=1$  by  Proposition \ref{p dividing n}. Hence $H^1(\M(\pmb{k}))=0$ and, consequently, $H^1(\M(\ker\pi))=0$. Along with Theorem \ref{injective H2}, the above exact sequence
also shows that $H^2(\M_0(H))\to H^2(\M(\ker\pi))$ is an injection. Hence $\M_0(H)\rtimes_x SL_n(W_B)=\varphi u H u^{-1}$ for some $u\in \pi^{-1}SL_n(W_B)$ (cf. sequence \ref{seq 1}) with $\pi(u)=I$ by Proposition \ref{coh prop}.

In any case, we have found a $u\in GL_n(A)$ with $\pi(u)=I$ and $\varphi u H u^{-1}=\M_0(H)\rtimes_x SL_n(W_B)$.
Finally,
 \[
 \varphi SL_n(W_A)= \M_0(SL_n(W_A))\rtimes_x  SL_n(W_B)\subseteq \M_0(H)\rtimes_x SL_n(W_B)\]
as $\M_0(SL_n(W_A))\subseteq \M_0(H)$ by our claim \ref{claim 1}, and the proposition follows.

We now establish the claim to complete the argument.
\begin{proof}[Proof of Claim \ref{claim 1}] There is nothing to prove if
$W_A\xrightarrow{\pi}W_B$ is an injection  (as $\M_0(SL_n(W_A))$ is then $0$). Therefore we may suppose that we have a
natural identification of $W_A\xrightarrow{\pi}W_B$ with  $W_{m+1}\to W_m$ for some integer $m\geq 1$, and consequently an
identification of $\M_0(SL_n(W_A))$ with $\M_0(\pmb{k})$. We will freely use these identifications in what follows.

As in the proof of the proposition, let $x\in H^2(\M_0(\pmb{k}))$ represent   the extension
\ref{exact seq for SL_n}
and let $y\in H^2(\M_0(H))$ represent the extension
\begin{equation*}
0\xrightarrow{}\M_0(H)\xrightarrow{j} H\xrightarrow{} SL_n(W_B)
\xrightarrow{}I.\end{equation*}
Then $x$ and $y$ represent the same cohomology class in $H^2(\M_0(\ker \pi))$. By Corollary \ref{main corollary},  there is a $z\in H^2(\M_0(\pmb{k})\cap \M_0(H))$ such that $x$ and $z$ (resp. $y$ and $z$) represent the same cohomology class in $H^2(\M_0(\pmb{k}))$ (resp. $H^2(\M_0(H))$).

Suppose the claim $\M_0(\pmb{k})\subseteq \M_0(H)$ is false. Then we must have
 $\M_0(\pmb{k})\cap\M_0(H)\subseteq\s$ by Lemma \ref{lemma1}. Now, if $\M_0(\pmb{k})\cap \M_0(H)=0$  then $x$ will be zero, contradicting  non-splitting of the extension \ref{exact seq for SL_n}.

Thus $\M_0(\pmb{k})\cap \M_0(H)$ must be a non-zero  submodule of $\s$, and we must therefore have   $p$ dividing $n$. Now  the image
 of $x$ in $H^2(\M_0(\pmb{k})/\s)$ represents the extension
 \[
 0\xrightarrow{}\M_0(\pmb{k})/\s\xrightarrow{j} SL_n(W_{m+1})/Z\xrightarrow{\mod{p^m}} SL_n(W_m)
\xrightarrow{}I.
 \]Since  this is non-split by Corollay \ref{non-split p|n}, the image of $x$ in $H^2(\V)$ is not $0$. This contradicts the fact that $x$ is itself in the image of $H^2(\s)\to H^2(\M_0(\pmb{k}))$.
\end{proof}

\begin{remark} It is well known that the mod-$p$ reduction map $SL_2(\Z/p^2\Z)\to SL_2(\Z/p\Z)$ has a homomorphic section when $p$ is $2$ or $3$. (See the exercises at the end of \cite[Chapter IV(3)]{serre}.) Thus the conclusion of the main theorem fails when $n=2$ and $\pmb{k}$ is $\F_2$ or $\F_3$.

The main theorem also fails when $n=2$ and $\pmb{k}=\F_5$. To see this, choose $0\neq \xi\in H^1( SL_2(\F_5), \M_0(\F_5))$ and consider the subgroup
\[ G:=\{(I+\epsilon \xi(A))A \quad | \quad A\in SL_2(\F_5)\}\]
of $SL_2(\F_5[\epsilon])$ where $\F_5[\epsilon]$ is the ring of dual numbers (so $\epsilon^2=0$). Clearly,  $G\mod\epsilon =SL_2(\F_5)$. If $G$ can be conjugated to $SL_2(\F_5)$ in $GL_2(\F_5[\epsilon])$ then the cocycle $\xi$ must vanish in $H^1(SL_2(\F_5),\M(\F_5))$. This cannot happen as the sequence $0\to \M_0(\F_5)\to \M(\F_5)\to \F_5\to 0$ splits.
\end{remark}

\begin{remark} \label{application} Fix a finite field $\pmb{k}$ satisfying assumption \ref{assumptions on k} and an integer $m\geq 1$. The main theorem then determines the universal deformation ring for $G:=SL_n(W_m)$ with standard representation completely. (See \cite{mazur1}, \cite{mazur2} for background on deformation of representations.)

To describe this fully, let $\rho:G\to SL_n(W_m)$ be the natural representation and set $\overline\rho:=\rho\mod{p}$. We work inside the category of complete local  Noetherian rings with residue field $\pmb{k}$ from here on.
Let $R$ be the universal deformation ring for deformations of $(G,\overline\rho)$ in
this category and let
$\rho_R:G\to GL_n(R)$ be the universal representation.

By universality, there is a morphism $\pi :R\to W_m$ such that $\pi\circ\rho_R$ is strictly equivalent to $\rho$. By our main theorem $X\rho_R(G)X^{-1}\supseteq SL_n(W_R)$ for some $X$ in $GL_n(R)$; here, we can insist that $X$ reduces to the identity modulo $\m_R$. Now $\pi|_{W_R}:W_R\to W_m$ along with
\[
|SL_n(W_m)|=|G|\geq |\rho_R(G)|\geq |SL_n(W_R)|\geq |SL_n(W_m)|\]
implies that $\pi|_{W_R}:W_R\to W_m$ is an isomorphism and that $X\rho_R(G)X^{-1}=SL_n(W_R)$. Replacing $\rho_R$ with the  strictly equivalent representation $X\rho_RX^{-1}$ if necessary, we can then assume that
$\rho_R:G\to GL_n(R)$ takes values in $SL_n(W_R)$. Writing $i : W_m\to W_R$ for the inverse to $\pi|_{W_R}$, we conclude that $i\circ \rho$ is strictly equivalent to $\rho_R$.

We will now verify that $\rho: G\to SL_n(W_m)$ is the universal deformation. So given
a lifting
$\rho_A: G\to GL_n(A)$  of $\overline\rho: G\to SL_n(\pmb{k})$, we need to show that there is a unique morphism $i_A: W_m\to A$ such that $i\circ \rho$ is strictly equivalent to $\rho_A$. Uniqueness comes for free (it has to send $1$ to $1$). For  existence, note that  by universality there is a morphism $\pi_A: R\to A$ such that
$\pi_A\circ \rho_R$ is strictly equivalent to $\rho_A$. It is then an easy check to see that $i_A:=\pi_A\circ i$ works.
\end{remark}

\section*{Acknowledgements}
The author thanks the referee for bringing to attention the work of Boston in \cite{mazur wiles} and  other helpful comments.

\bibliographystyle{abbrv}

\end{document}